\documentclass[reqno,11pt]{amsart}
\usepackage[letterpaper, margin=1in]{geometry}
\RequirePackage{amsmath,amssymb,amsthm,graphicx,mathrsfs,url,slashed,subcaption,empheq}
\RequirePackage[usenames,dvipsnames]{xcolor}
\RequirePackage[colorlinks=true,linkcolor=Red,citecolor=Green]{hyperref}
\RequirePackage{amsxtra}
\usepackage{cancel}
\usepackage{tikz, quiver, wrapfig}

\makeatletter
\@namedef{subjclassname@2020}{\textup{2020} Mathematics Subject Classification}
\makeatother

\title{Reconstructing currents from their projections}
\author{Aidan Backus}
\address{Department of Mathematics, Brown University}
\email{aidan\_backus@brown.edu}
\date{\today}
\thanks{This research was supported by the National Science Foundation's Graduate Research Fellowship Program under Grant No. DGE-2040433.}

\newcommand{\NN}{\mathbf{N}}

\newcommand{\RR}{\mathbf{R}}

\newcommand{\Sph}{\mathbf S}

\newcommand*\dif{\mathop{}\!\mathrm{d}}

\DeclareMathOperator{\dist}{dist}

\newcommand{\vol}{\mathrm{vol}}

\newcommand{\diam}{\mathrm{diam}}

\newcommand{\dfn}[1]{\emph{#1}\index{#1}}

\newtheorem{theorem}{Theorem}

\newtheorem{lemma}[theorem]{Lemma}

\theoremstyle{definition}
\newtheorem{definition}[theorem]{Definition}



\def\XXint#1#2#3{{\setbox0=\hbox{$#1{#2#3}{\int}$ }
\vcenter{\hbox{$#2#3$ }}\kern-.6\wd0}}

\usepackage[backend=bibtex,style=alphabetic,giveninits=true]{biblatex}

\renewbibmacro{in:}{}
\DeclareFieldFormat{pages}{#1}

\begin{document}
\begin{abstract}
We prove an inversion formula for the exterior $k$-plane transform.
As a consequence, we show that if $m < k$ then an $m$-current in $\RR^n$ can be reconstructed from its projections onto $\RR^k$, which proves a conjecture of Solomon.
\end{abstract}

\maketitle

A basic problem in analysis and geometry is to reconstruct an object in euclidean space from its orthogonal projections.
Here we are interested in reconstructing an oriented submanifold, or more generally a current, from its projections.
The geometry needed to prove our main theorem was worked out by Solomon \cite{Solomon11}, who conjectured that the analysis would work out as well.
In this brief note, we prove that conjecture.

Throughout, fix integers $0 \leq m < k < n$.
A \dfn{compactly supported $m$-current} is a continuous linear functional on the space of smooth $m$-forms \cite[\S6.2]{simon1984lectures}.
If $T$ is a compactly supported $m$-current and $\alpha$ is a smooth $m$-form, we write $\int_T \alpha$ for $T(\alpha)$.
Thus every compact, oriented, $C^1$ submanifold of $\RR^n$ is a compactly supported $m$-current.

Let $G$ be the Grassmannian variety of $k$-dimensional linear subspaces of $\RR^n$.
There is a natural $O(n)$-action on $G$, which induces a unique $O(n)$-invariant Borel probability measure on $G$.
For any $P \in G$, we have two operations:
\begin{enumerate}
\item Given an $m$-form $\alpha$ on $P$, the \dfn{pullback} $P^* \alpha$ is the $m$-form on $\RR^n$ which is the pullback of $\alpha$ by the orthogonal projection, $\RR^n \to P$.
\item Given a compactly supported $m$-current $T$ on $\RR^n$, the \dfn{pushforward} $P_* T$ is the compactly supported $m$-current on $P$ such that for every $m$-form $\alpha$ on $P$,
$$\int_{P_* T} \alpha = \int_T P^* \alpha.$$
\end{enumerate}
It is very important to work in the category of compactly supported currents for the pushforward by an orthogonal projection to make sense.
Indeed, one can only push forward a noncompactly supported current by a proper map, which orthogonal projections are not; moreover, it does not really make sense to make sense to push forward a submanifold of $\RR^n$.

We take the convention $0 \in \NN$.
Let $\langle x\rangle := \sqrt{1 + |x|^2}$ be the Japanese bracket.
By $A \lesssim_k B$ we mean that there exists $C > 0$ (which depends on $k$) for every $A, B$, $A \leq CB$.
The \dfn{Schwartz space} is the Fr\'echet space $\mathscr S$ of $m$-forms $\alpha$ such that for every $N \in \NN$, $|\alpha(x)| \lesssim_N \langle x\rangle^{-N}$ \cite[\S7.1]{hörmander2015analysis}.
A density argument shows that a compactly supported $m$-current is determined by its action on $\mathscr S$.

Our main theorem establishes \cite[Conjectures 6.5 and 6.8]{Solomon11}:

\begin{theorem}\label{thm: decomposition of currents}
For every $\alpha \in \mathscr S$ and $P \in G$, there exists a smooth $m$-form $\alpha_P$ on $P$ such that for every $x \in \RR^n$,
\begin{equation}\label{eqn: decomposition of forms}
\alpha(x) = \int_{G} P^*(\alpha_P)(x) \dif P,
\end{equation}
and for every compactly supported $m$-current $T$ on $\RR^n$, 
\begin{equation}\label{eqn: decomposition of currents}
\int_T \alpha = \int_{G} \int_{P_* T} \alpha_P \dif P.
\end{equation}
\end{theorem}

In order to decompose $\alpha$ into $\alpha_P$s, Solomon constructed a Radon-like transform \cite[\S3]{Solomon11}.
Let $\Gamma \to G$ be the \dfn{tautological bundle}, the vector bundle whose fiber over $P \in G$ is $P$ itself.
For a vector space $Q$, let $\Lambda^m(Q)$ be the $m$th exterior power of $Q$, so that $m$-forms on $Q$ are mappings $Q \to \Lambda^m(Q)$.
For any $P \in G$, the inclusion map $P \to \RR^n$ induces a linear map $\Lambda^m(P) \to \Lambda^m(\RR^n)$, and we denote its adjoint by 
\begin{align*}
\Lambda^m(\RR^n) &\to \Lambda^m(P) \\
v &\mapsto v|_P.
\end{align*}

\begin{definition}[{\cite[Definition 3.2]{Solomon11}}]
The \dfn{exterior $k$-plane transform} of an $m$-form $\alpha$ is the mapping $\Gamma \to \Lambda^m(\RR^n)$ defined by
\begin{equation}\label{eqn: exterior plane transform}
\mathcal R \alpha(P, x) := \int_{P^\perp} \alpha(x + y)|_P \dif y.
\end{equation}
\end{definition}

In (\ref{eqn: exterior plane transform}), $\int_{P^\perp} \alpha(x + y)|_P \dif y$ is not the integral in the sense of differential forms (indeed, it cannot be, since $\alpha|_P$ is valued in $\Lambda^m(P)$ but we are integrating it over $P^\perp$).
Rather, it is an $\Lambda^m(P)$-valued integral.
Moreover, (\ref{eqn: exterior plane transform}) only makes sense when for almost every $P$, $\alpha|_P$ is integrable on $P^\perp$, so $\mathcal R$ is ill-defined in $L^p$ when $p$ is large \cite{Christ84}.
We study $\mathcal R$ on the completion $X$ of $\mathscr S$ under the norm
$$\|\alpha\|_X^2 := \|\alpha\|_{L^2}^2 + \|\mathcal R\alpha\|_{L^2(\Gamma)}^2.$$

To prove Theorem \ref{thm: decomposition of currents}, we must establish the inversion formula for $\mathcal R$, so now we recall Solomon's construction of the inverse of $\mathcal R$.
Let $\mathcal F$ denote the Fourier transform, normalized
$$\mathcal Ff(\xi) := \int_{\RR^n} e^{-2\pi ix\cdot \xi} f(x) \dif x.$$
For each $\xi \in \RR^n \setminus \{0\}$, let $H_\xi$ be the hyperplane orthogonal to $\xi$, and let
$$\Pi(\xi): \Lambda^m(\RR^n) \to \Lambda^m(H_\xi)$$
be the orthogonal projection.
Let $\Psi(\xi)$ be the orthogonal projection onto the orthocomplement of $\Lambda^m(H_\xi)$ in $\Lambda^m(\RR^n)$, so that $\Pi(\xi) + \Psi(\xi) = I$, and introduce the $0$-homogeneous symbol
\begin{align*}
h: \RR^n \setminus \{0\} &\to GL(\Lambda^m(\RR^n)) \\
\xi &\mapsto \frac{k \vol(\Sph^{n - 1}) \binom{n - 1}{m}}{\vol(\Sph^{k - 1}) \binom{n - k}{m}} \left[\frac{\Pi(\xi)}{k - m} + \frac{\Psi(\xi)}{n - m}\right]
\end{align*}
and the Fourier multiplier,
$$\mathcal Q \gamma(x) := \int_{\RR^n} e^{2\pi ix \cdot \xi} |\xi|^{n - k} h(\xi) \mathcal F\gamma(\xi) \dif \xi.$$

\begin{lemma}[{\cite[\S3.2]{Solomon11}}]\label{lma: properties of dual transform}
The formal adjoint of $\mathcal R$, applied to a mapping $\beta: \Gamma \to \Lambda^m(\RR^n)$, is 
$$\mathcal R^* \beta(x) = \int_{G} \beta(P, x_P)|_P \dif P,$$
where $x_P$ is the orthogonal projection of $x$ onto $P$.
The restriction, $\alpha_P$, of $\beta(P, x_P)|_P$ to $P$ is an $m$-form on $P$ such that 
$$P^*(\alpha_P)(x) = \beta(P, x_P)|_P.$$
\end{lemma}

\begin{theorem}[inversion formula]\label{thm: Radon inversion formula}
For every $\alpha \in X$,
\begin{equation}\label{eqn: Radon inversion formula}
\alpha = \mathcal Q \mathcal R^* \mathcal R\alpha.
\end{equation}
\end{theorem}
\begin{proof}
By \cite[Theorem 6.1]{Solomon11}, (\ref{eqn: Radon inversion formula}) holds if $\alpha \in \mathscr S$.
Let $\dot H^s$ be the homogeneous Sobolev space of order $s$.
We claim that for every $\alpha \in X$, 
\begin{equation}\label{eqn: dual Radon transform bounded}
\|\mathcal R^* \mathcal R \alpha\|_{\dot H^{n - k}} \lesssim \|\alpha\|_{L^2}.
\end{equation}
It is enough to prove this when $\alpha \in \mathscr S$, so that (\ref{eqn: Radon inversion formula}) holds; thus, since $h(\xi)$ is invertible,
$$|\xi|^{n - k} \mathcal F(\mathcal R^* \mathcal R\alpha)(\xi) = h(\xi)^{-1} \mathcal F \alpha(\xi).$$
Thus (\ref{eqn: dual Radon transform bounded}) follows from Plancherel's theorem.

To prove (\ref{eqn: Radon inversion formula}), let $\alpha \in X$ and let $\alpha_n \in \mathscr S$ satisfy $\alpha_n \to \alpha$ in $X$.
Since $h$ is $0$-homogeneous, $\mathcal Q$ maps $\dot H^{n - k}$ boundedly into $L^2$, so since (\ref{eqn: Radon inversion formula}) holds for $\alpha_n$,
\begin{align*}
\|\mathcal Q \mathcal R^* \mathcal R \alpha - \alpha\|_{L^2} 
&\leq \|\mathcal Q \mathcal R^* \mathcal R (\alpha_n - \alpha)\|_{L^2} + \|\alpha_n - \alpha\|_{L^2} \\
&\lesssim \|\mathcal R^* \mathcal R (\alpha_n - \alpha)\|_{\dot H^{n - k}} + \|\alpha_n - \alpha\|_{L^2} \\
&\lesssim \|\alpha_n - \alpha\|_{L^2}. \qedhere
\end{align*}
\end{proof}

\begin{lemma}\label{lma: decaying forms are good}
Suppose that $\alpha$ is a measurable $m$-form such that $|\alpha(x)| \lesssim \langle x\rangle^{-n}$.
Then $\alpha \in X$.
\end{lemma}
\begin{proof}
Let $(\alpha_j)$ be a sequence in $\mathscr S$ such that $|\alpha_j(x)| \lesssim \langle x\rangle^{-n}$ uniformly in $j$, and $\|\alpha_j - \alpha\|_{L^2} \leq 1/j$.
Such a sequence can be constructed, for example, by taking a smooth cutoff of $\alpha$, then running the heat equation for a short time; this is routine, and we omit the details.
To show that $\alpha \in X$, we must show that $\|\mathcal R(\alpha_j - \alpha)\|_{L^2(\Gamma)} \to 0$.
Let $\varepsilon \in (0, 1/10]$ and $R := \varepsilon^{-4}$.
For every $x, y \in \RR^n$ such that $x \cdot y = 0$, $|x + y| \geq \max(|x|, |y|)$, so
$$\langle x + y\rangle^{-n} \leq \langle x\rangle^{-k + 1/4} \langle y\rangle^{-(n - k)-1/4};$$
we use this fact to estimate
\begin{align*}
&\int_G \int_P \left|\int_{P^\perp \cap \{|y| > R\}} (\alpha(x + y) - \alpha_j(x + y))|_P \dif y\right|^2 \dif x \dif P\\
&\qquad \lesssim \int_G \int_P \left|\int_{P^\perp \cap \{|y| > R\}} \langle x + y\rangle^{-n} \dif y\right|^2 \dif x \dif P\\
&\qquad \leq \int_G \int_P \langle x\rangle^{-2k + 1/2} \left|\int_{P^\perp \cap \{|y| > R\}} \langle y\rangle^{-(n - k)-1/4} \dif y\right|^2 \dif x \dif P \\
&\qquad \lesssim \int_R^\infty r^{-(n - k) - 1/4 + (n - k - 1)} \dif r \\
&\qquad \lesssim \varepsilon.
\end{align*}
Now we restrict to $j \geq R^{k/2}/\varepsilon^{1/2}$; then, by Jensen's inequality,
\begin{align*}
&\int_G \int_P \left|\int_{P^\perp \cap \{|y| \leq R\}} (\alpha(x + y) - \alpha_j(x + y))|_P \dif y\right|^2 \dif x \dif P\\
&\qquad \lesssim \int_G \int_P R^k \int_{P^\perp \cap \{|y| \leq R\}} |\alpha(x + y) - \alpha_j(x + y)|^2 \dif y \dif x \dif P \\
&\qquad \leq R^k \int_{\RR^n} |\alpha(z) - \alpha_j(z)|^2 \dif z \\
&\qquad \leq \varepsilon.
\end{align*}
Then $\|\mathcal R(\alpha_j - \alpha)\|_{L^2(\Gamma)} \lesssim \varepsilon$, as required.
\end{proof}

\begin{lemma}\label{lma: sufficient decay}
$\mathcal Q$ maps $\mathscr S$ boundedly into $X$.
Furthermore, for every $\alpha \in \mathscr S$, $\theta \in [0, 1)$, $s \in \NN$, $P, Q \in G$, and $x \in \RR^n$,
\begin{equation}\label{eqn: derivatives of the Radon transform}
\left|\frac{\partial^s}{\partial x^s} \mathcal R\mathcal Q\alpha(P, x_P) - \frac{\partial^s}{\partial x^s} \mathcal R\mathcal Q\alpha(Q, x_Q)\right| \lesssim_{\alpha, \theta, s} \dist(P, Q)^\theta.
\end{equation}
\end{lemma}
\begin{proof}
For every multiindex $A$ of order $s$,
\begin{equation}\label{eqn: derivatives of Q}
\frac{\partial^A}{\partial x^A} \mathcal Q\alpha(x) = (2\pi i)^s \lim_{\varepsilon \to 0} \int_{\{\xi > \varepsilon\}} e^{2\pi ix\cdot \xi} \xi^A h(\xi) |\xi|^{n - k} \mathcal F\alpha(\xi) \dif \xi.
\end{equation}
Since $e^{2\pi ix\cdot \xi} = \partial_{\xi_j} e^{2\pi ix \cdot \xi}/(2\pi ix_j)$, integration by parts and induction on $n$ gives, for each $j$,
\begin{align*}
&\left|\int_{\{|\xi| > \varepsilon\}} e^{2\pi ix\cdot \xi} \xi^A h(\xi) |\xi|^{s + n - k} \mathcal F\alpha(\xi) \dif \xi\right| \\
&\qquad \lesssim |x_j|^{-n} \left|\int_{\{|\xi| > \varepsilon\}} e^{2 \pi ix\cdot \xi} \left(\frac{\partial}{\partial \xi_j}\right)^n \left(p_{1, s + n - k}(\xi) \mathcal F\alpha(\xi)\right) \dif \xi\right| \\
&\qquad \qquad + |x_j|^{-n} \left|\int_{\{|\xi| = \varepsilon\}} e^{2\pi i x \cdot \xi} p_{1, 0}(\xi) \left(\frac{\partial}{\partial \xi_j}\right)^{n - 1} \left(p_{2, s + n - k}(\xi) \mathcal F\alpha(\xi)\right) \dif \xi\right| \\
&=: |x_j|^{-n} (I_1 + I_2)
\end{align*}
where $p_{\sigma, \tau}$ denotes a $\tau$-homogeneous function.
The gradient of a $\tau$-homogeneous function is $\tau - 1$-homogeneous, so for every $N \in \NN$,
\begin{align*}
I_1
&\lesssim \sum_{t=1}^n \int_{\{|\xi| > \varepsilon\}} |\nabla^t p_{1, s + n - k}(\xi)| \cdot |\nabla^{n - t} \mathcal F\alpha(\xi)| \dif \xi \\
&\lesssim_{N, \alpha} \sum_{t=1}^n \int_{\{|\xi| > \varepsilon\}} |\xi|^{s + n - k - t} \cdot \langle \xi\rangle^{-N} \dif \xi \\
&\lesssim \sum_{t = 1}^n \int_\varepsilon^\infty r^{s + n - k - t + n - 1} \cdot \langle r\rangle^{-N} \dif r.
\end{align*}
We can choose $N$ large enough (depending on $s$) so that each term in this sum is $\lesssim_s 1$.
Indeed, the worst case scenario is that $(t, s) = (n, 0)$, which is acceptable since $n - k \geq 1$.
Similarly,
\begin{align*}
I_2
&\lesssim \sum_{t=1}^{n - 1} \int_{\{|\xi| = \varepsilon\}} |\nabla^t p_{2, s + n - k}(\xi)| \cdot |\nabla^{n - t - 1} \mathcal F\alpha(\xi)| \dif \xi \\
&\lesssim_\alpha \sum_{t =1}^{n - 1} \int_{\{|\xi| = \varepsilon\}} |\xi|^{s + n - k - t} \dif \xi \\
&\lesssim \sum_{t =1}^{n - 1} \varepsilon^{s + n - k - t + n - 1}
\end{align*}
and again all terms are acceptable.
So by (\ref{eqn: derivatives of Q}), $|\nabla^s \mathcal Q\alpha(x)| \lesssim_{\alpha, s} \min_j |x_j|^{-n}$; (\ref{eqn: derivatives of Q}) also easily implies that $|\nabla^s \mathcal Q\alpha(x)| \lesssim_{\alpha, s} 1$.
Therefore
\begin{equation}\label{eqn: rapid decay of derivatives}
|\nabla^s \mathcal Q\alpha(x)| \lesssim_{\alpha, s} \langle x\rangle^{-n}.
\end{equation}
Taking $s = 0$, and applying Lemma \ref{lma: decaying forms are good}, we obtain $\mathcal Q \alpha \in X$.
In particular, $\mathcal R \mathcal Q \alpha$ exists, so (\ref{eqn: derivatives of the Radon transform}) makes sense.

To prove (\ref{eqn: derivatives of the Radon transform}), we may restrict to the case that $P, Q$ are contained in a small ball $G_0 \subset G$.
In particular, there exists a smooth family of linear isometries, $\sigma_N: \RR^{n - k} \to \RR^n$, parametrized by $N \in G_0$, such that the image of $\sigma_N$ is $N^\perp$.
Consider the mapping
\begin{align*}
\gamma: G_0 \times \RR^n &\to \Lambda^m(\RR^n) \\
(N, x) &\mapsto \mathcal R\mathcal Q\alpha(N, x_N).
\end{align*}
Let $\tau_N: \RR^n \to N$ be the orthogonal projection $x \mapsto x_N$, so
$$\frac{\partial^s \gamma}{\partial x^s}(N, x) 
= \int_{\RR^{n - k}} ((\nabla^s \mathcal Q \alpha)(\tau_N x + \sigma_N y)|_N) \tau_N^{\otimes s} \dif y$$
and therefore
\begin{align*}
\left|\frac{\partial^s \gamma}{\partial x^s}(P, x) - \frac{\partial^s \gamma}{\partial x^s}(Q, x)\right| 
&\leq \int_{\RR^{n - k}} |(\nabla^s \mathcal Q\alpha)(\tau_P x + \sigma_P y) - (\nabla^s \mathcal Q\alpha)(\tau_Q x + \sigma_Q y)| \dif y \\
&\qquad + \int_{\RR^{n - k}} |(\nabla^s \mathcal Q\alpha)(\tau_Q x + \sigma_Q y)| \cdot (||_P - |_Q| + |\tau_P^{\otimes s} - \tau_Q^{\otimes s}|) \dif y \\
&=: J_1 + J_2.
\end{align*}
We estimate using (\ref{eqn: rapid decay of derivatives}), the $\theta$-H\"older continuity of $N \mapsto (\tau_N, \sigma_N)$, and the fact that $\theta < 1 \leq k$,
\begin{align*}
J_1 
&\lesssim_{\alpha, s} \int_{\RR^{n - k}} \langle \tau_P x + \sigma_P y\rangle^{-n} \cdot |(\tau_P - \tau_Q) x + (\sigma_P - \sigma_Q) y| \dif y \\
&\lesssim \dist(P, Q)^\theta \int_{\RR^{n - k}} \langle \tau_P x + \sigma_P y\rangle^{-n} |\tau_P x + \sigma_P y|^\theta \dif y \\
&\lesssim_\theta \dist(P, Q)^\theta.
\end{align*}
A similar and even easier calculation shows $J_2 \lesssim_{\alpha, s} \dist(P, Q)^\theta$ and therefore (\ref{eqn: derivatives of the Radon transform}).
\end{proof}

\begin{proof}[Proof of Theorem \ref{thm: decomposition of currents}]
By Lemma \ref{lma: sufficient decay}, $\mathcal Q\alpha \in X$, so $\beta := \mathcal R \mathcal Q\alpha$ is well-defined.
By Theorem \ref{thm: Radon inversion formula} and the Fourier inversion formula,
$$|\xi|^{n - k} h(\xi) \mathcal F \alpha(\xi) = \mathcal F\mathcal Q\alpha(\xi) = |\xi|^{n - k} h(\xi) \mathcal F(\mathcal R^* \beta)(\xi).$$
Since $h(\xi)$ is invertible, we obtain $\alpha = \mathcal R^* \beta$.

For each $P \in G$, let $\alpha_P$ be the restriction of $\beta(P, \cdot)|_P$ to $P$.
By Lemma \ref{lma: properties of dual transform}, $\alpha_P$ is an $m$-form on $P$ such that for each $x \in \RR^n$,
\begin{equation}\label{eqn: superposition from Radon}
\beta(P, x_P)|_P = \alpha_P(x_P) = P^*(\alpha_P)(x).
\end{equation}
In particular, 
$$\alpha(x) = \mathcal R^* \beta(x) = \int_{G} \beta(P, x_P)|_P \dif P = \int_{G} P^*(\alpha_P)(x) \dif P,$$
establishing (\ref{eqn: decomposition of forms}).

The sets $\{\gamma \in C^\infty: \|\gamma\|_{C^\ell(B_R)} < \varepsilon\}$ generate the filter $\mathscr F$ of open neighborhoods of $0$ in $C^\infty$.
So by differentiating both sides of (\ref{eqn: decomposition of forms}) repeatedly, and viewing the right-hand side of (\ref{eqn: decomposition of forms}) as a limit of Riemann sums, we see that for every $\mathcal U \in \mathscr F$ and $\varepsilon > 0$, there is a partition $E_1^{\mathcal U, \varepsilon}, \dots, E_{I^{\mathcal U, \varepsilon}}^{\mathcal U, \varepsilon}$ of $G$ into Borel sets such that $\diam(E_i^{\mathcal U, \varepsilon}) < \varepsilon$ and whenever $P_i^{\mathcal U, \varepsilon} \in E_i^{\mathcal U, \varepsilon}$,
\begin{equation}\label{eqn: Riemann sum error is small in Cinfty}
\left(\alpha - \sum_{i=1}^{I^{\mathcal U, \varepsilon}} \vol(E_i^{\mathcal U, \varepsilon}) (P_i^{\mathcal U, \varepsilon})^*(\alpha_{P_i^{\mathcal U, \varepsilon}})\right) \in \mathcal U.
\end{equation}
For every compactly supported $m$-current $T$ and $\varepsilon > 0$, there exists $\mathcal U_\varepsilon \in \mathscr F$ such that for every $\gamma \in \mathcal U_\varepsilon$, $|\int_T \gamma| < \varepsilon$, so by (\ref{eqn: Riemann sum error is small in Cinfty}),
\begin{equation}\label{eqn: Riemann sum approximation is good}
\left|\int_T \alpha - \sum_{i=1}^{I^{\mathcal U_\varepsilon, \varepsilon}} \vol(E_i^{\mathcal U_\varepsilon, \varepsilon}) \int_T (P_i^{\mathcal U_\varepsilon, \varepsilon})^*(\alpha_{P_i^{\mathcal U_\varepsilon, \varepsilon}})\right| < \varepsilon.
\end{equation}
By (\ref{eqn: derivatives of the Radon transform}) and (\ref{eqn: superposition from Radon}), $P \mapsto \int_T P^*(\alpha_P)$ is continuous, so since $\diam(E_i^{\mathcal U_\varepsilon, \varepsilon}) < \varepsilon$, the Riemann sums converge:
$$\lim_{\varepsilon \to 0} \sum_{i=1}^{I^{\mathcal U_\varepsilon, \varepsilon}} \vol(E_i^{\mathcal U_\varepsilon, \varepsilon}) \int_T (P_i^{\mathcal U_\varepsilon, \varepsilon})^*(\alpha_{P_i^{\mathcal U_\varepsilon, \varepsilon}}) = \int_{G} \int_T P^*(\alpha_P) \dif P = \int_G \int_{P_* T} \alpha_P \dif P.$$
Combining this equation with (\ref{eqn: Riemann sum approximation is good}) completes the proof of (\ref{eqn: decomposition of currents}).
\end{proof}

When $m = 0$, sharper conditions are available for the invertibility of $\mathcal R$ \cite{Rubin2004, Jensen_2004}, and it is natural to conjecture the same when $m \geq 1$.
Since they are necessary for the proof of Theorem \ref{eqn: decomposition of currents}, I have not seriously attempted to generalize them.

\printbibliography

\end{document}